\documentclass[10pt]{article}

\usepackage{amsfonts}
\usepackage{amsthm}
\usepackage{graphics}
\usepackage{epsfig}

\usepackage{graphicx}
\usepackage{psfrag}

\newtheorem{thm}{Theorem}[section]
\newtheorem{lem}[thm]{Lemma}
\newtheorem{prop}[thm]{Proposition}

\theoremstyle{definition}
        \newtheorem{dfn}[thm]{Definition}
\theoremstyle{definition}
        \newtheorem{rem}[thm]{Remark}

\newcommand{\N}{\mathbb{N}}

\newcommand{\R}{\mathbb{R}}

\newenvironment{pf*}[1]{\noindent\mbox{{\em
{#1}}.}}{\\\hspace*{\fill}$\Box$\\\medskip}

\def\hpic #1 #2 {\mbox{$\begin{array}[c]{l} \epsfig{file=#1,height=#2}
\end{array}$}}
\def\frac#1#2{{#1\over#2}}

\begin{document}

\title{Evasive Random Walks and the Clairvoyant Demon\footnote{This article was published in \emph{Random Structures \& Algorithms,} vol.~20 no.~2 (2002), pp.~239--248.}}
\author{Aaron Abrams\footnote{Math Department, UC Berkeley, Berkeley CA 
94720, {\tt abrams@math.berkeley.edu}}, 
Henry Landau\footnote{{\tt hjl@research.bell-labs.com}},
Zeph Landau\footnote{Sloan Center for Theoretical Neurobiology, 
Dept. of Physiology, Box 0444, University of California, San Francisco, 
513 Parnassus Avenue, San Francisco CA, 94122, {\tt landau@math.berkeley.edu}},
\\
James Pommersheim\footnote{Department of Mathematics, Pomona College,
610 N. College Ave., Claremont CA 91711-6348, {\tt jpommersheim@POMONA.EDU}},
and Eric Zaslow\footnote{Department of Mathematics, Northwestern University,
2033 Sheridan Road, Evanston, IL 60208, {\tt zaslow@math.nwu.edu}}}

\date{}
\maketitle

\begin{abstract}
A pair of random walks $(R,S)$ on the vertices of a graph $G$ is {\it
successful} if two tokens can be scheduled (moving only one token at a 
time) to travel along $R$ and $S$ without colliding.  We consider 
questions related to P. Winkler's {\it clairvoyant demon problem}, 
which asks whether for random walks $R$ and $S$ on $G$,
$Pr[\ (R,S) \mbox{ is successful }] >0$.  
We introduce the notion of an {\it evasive} walk on $G$:  a walk $S$ 
so that for a random walk $R$ on $G$, $Pr[\ (R,S) \mbox{ is successful }]>0$.
We characterize graphs $G$ having evasive walks, giving explicit 
constructions on such $G$.  On a cycle, we show that with high 
probability the tokens must collide quickly.  Finally we consider two 
variants of the problem for which, under certain assumptions on the 
graph $G$, we provide algorithms that schedule $(R,S)$ successfully with 
positive probability.

\end{abstract}

\section{Introduction}

In 1993, Coppersmith, Tetali, and Winkler \cite{CTW} considered
the following question about random walks on graphs.  Suppose two tokens
are taking random walks on a (finite, simple, connected) graph $G$; this 
means that each token 
sits at a vertex of $G$ and, when instructed to move, advances to a
(uniformly) randomly chosen adjacent vertex.  A scheduling demon, whose
goal is to keep the tokens from colliding, chooses at each time step which
one of the tokens is to move.  The question, then, is can the demon 
prevent the tokens from ever colliding?  The main theorem of Coppersmith,
Tetali, and Winkler
is that if the demon has no knowledge other
than the current locations of the tokens, then regardless of
the demon's strategy, the tokens will collide 
in expected time at most $(4/27+o(1))n^3$, where $n$ is the number of vertices 
of $G$.

This result led Winkler to ask whether the demon can do better with
more information.  Specifically, suppose the demon knows from the outset 
which (infinite) path will be travelled by each token.  Can this 
{\it clairvoyant demon} succeed at preventing a collision forever?  

\begin{dfn}
Let $G$ be a graph, and let $R$ and $S$ be two (infinite) walks on $G$.
If two tokens $T_R$ and $T_S$ can be scheduled to travel along $R$ and $S$ 
without colliding, we call $(R,S)$ a {\em successful pair} of walks.  
\end{dfn}

For example, if $G$ is the complete graph on the vertices $A$, $B$, $C$, $D$ 
then any pair $(R,S)$ where $R$ begins $AB....$ and $S$ begins 
$BA....$ is not successful, as the demon has no first move.  By contrast 
the pair $ABABAB...$ and $CDCDCD...$
is obviously successful.  Note that in general, as in the first example,
if a pair is not successful then this can be detected by looking at some 
finite initial strings of the two walks.  Thus in particular there is a
positive chance that a random pair is \emph{not} successful.  Detecting success
of a pair, however, is not so easy, and it is by no means obvious that
the probability of a pair's being successful should be positive.

\begin{dfn}
We say that the {\em clairvoyant demon succeeds on $G$} if for randomly 
chosen walks $R$ and $S$ on $G$, $(R,S)$ is a successful pair with positive 
probability.
\end{dfn}

Thus the Clairvoyant Demon problem is to determine, for a graph $G$,
whether the clairvoyant demon succeeds on $G$.  It is obvious that 
the demon does not succeed on a path; Winkler \cite{fonz} has
shown that the demon doesn't succeed on a cycle either.
It is not known whether there is any graph on which
the demon succeeds, but Winkler has conjectured that the demon succeeds
on the complete graph $K_4$.

The clairvoyant demon problem has a natural reformulation in the 
language of percolation; see \cite{fonz}.  From this point of view,
Winkler \cite{fonz} and independently Balister, Bollob\'as, and 
Stacey \cite{BBS} have shown that the {\it fickle demon}, who is allowed 
to move the tokens both forward and backward along their walks, does
succeed on $K_4$.  Also G\'acs \cite{gacs} has shown that on a graph
where the clairvoyant demon succeeds, the probability of being able to
advance the tokens for $n$ steps but not indefinitely is bounded below
by $n^{-a}$, for some $a>0$, rather than decreasing exponentially as
might have been expected.

In this paper we rephrase the clairvoyant demon problem 
as a problem about single random walks, rather than pairs.  Specifically,
let $S$ be a fixed walk.  If, when $R$ is chosen randomly, the pair
$(R,S)$ is successful with positive probability, then we call $S$ 
\emph{evasive}.  If a random choice of $S$ is evasive with positive 
probability, then the clairvoyant demon succeeds on $G$.  We therefore
focus on evasive walks.  

The main result of this paper is

\begin{thm} \label{main} 
There is an evasive walk on $G$ if and only if $G$ is not a path, a 
cycle, or the graph $K_{1,3}$ consisting of one vertex of degree 3 and 
three pendant vertices.
\end{thm}

\noindent On a cycle, we also obtain an upper bound 
on the chance of advancing $N$ steps without a collision; this result 
can be viewed as a partial converse to G\'acs's result.

Section \ref{sec.evasive} is devoted to constructions of evasive walks 
with various desirable properties.  In Section \ref{sec.cycle} we obtain 
the bound mentioned
above, which in particular implies that there are no
evasive walks on a cycle, and hence the demon doesn't succeed
on any cycle.  (This will also complete the proof of Theorem 
\ref{main}.)
Section \ref{sec.variants} addresses two variants of the
clairvoyant demon problem which have interpretations in the areas of
traffic control and robotics.

\section{Evasive walks}\label{sec.evasive}

By a walk $R$ on $G$ we mean a sequence $R(1), R(2), \ldots$ of vertices
of $G$ such that $R(i)$ is adjacent to $R(i+1)$ for each $i$.  We also 
sometimes think of $R$ as a function from $\N$ to the vertex set of $G$,
using notation such as $R^{-1}(v)$ to denote the set of ``times" when
$R$ visits the vertex $v$.  

\begin{dfn}
A walk $R$ on $G$ is {\em evasive} if for a randomly chosen walk $S$ on 
$G$ the pair $(R,S)$ is successful with positive probability.
\end{dfn}

\begin{rem}\label{evasivemeasure}
By Fubini's theorem, the clairvoyant demon succeeds on $G$ if and only 
if the set of evasive walks on $G$ has positive measure in the
set of all walks on $G$.
\end{rem}

Thus the clairvoyant demon problem is equivalent to asking if a random
walk on $G$ is evasive with positive probability.  In this section we
give several constructions of evasive walks on graphs, though we remain
unable to construct a set of such walks of positive measure on any graph.

We will also have occasion to think about finite walks.

\begin{dfn}
Let $R$ and $S$ be finite walks on $G$.  If tokens $T_R$ and $T_S$
can be advanced to the end of $R$ and $S$ without colliding, 
then we call each of $R$, $S$ an {\em allowed walk} for the other.
\end{dfn}

If $G$ is a path, the walk $R$ which traverses $G$ from end to end
has no allowed walks.  It is not hard to see that this is the only such case:

\begin{rem} \label{avoiding}
If $H$ is a connected subgraph of $G$ that is not a path, then every 
finite walk $R$ on $G$ has an allowed walk on $H$.
\end{rem}

The {\em length} of a finite walk is the number of vertices it visits 
(counted with multiplicity); this is one more than the usual definition of 
length, which is the number of edges traversed.  

If $G$ is a graph and $v$ a vertex of $G$, then $G-v$ denotes the
graph obtained from $G$ by deleting $v$ and all edges incident with $v$.
(This is the subgraph induced by the vertices of $G$ other than $v$.)

Finally, $K_n$ denotes a complete graph on $n$ vertices.

\subsection{Constructions}

Our first result applies in particular to the graph $K_4$.

\begin{thm}\label{K40}  Let $G$ be a (finite, simple, connected) graph
which is not a path, a cycle, or the graph $K_{1,3}$.
Then there exist uncountably many evasive walks 
on $G$.
\end{thm}

For the proof we record the following characterization.

\begin{rem}\label{characterization.1}
Let $G$ be a (finite, simple, connected) graph.  Then $G$ is not a path, 
a cycle, or $K_{1,3}$ if and only if there exists a vertex $v$ of $G$ such 
that some component of $G-v$ is not a path.
\end{rem}

\begin{pf*}{Proof of Theorem \ref{K40}}  
Choose $v$ according to Remark \ref{characterization.1},
and let $G'$ be a component of $G-v$ which is not a path.  Choose $S$ 
to be any walk on $G'$ with the property that each finite walk on $G'$ 
occurs as a subwalk of $S$ infinitely 
many times.  (A random walk on $G'$ has this property with probability 
one.)  We will show that $S$ is evasive.

Let $n$ be the number of vertices of $G$; then with probability $1/n$ a
random walk $R$ on $G$ has $R(1)=v$ and 
$R^{-1}(v)$ infinite.  Choosing  $R$ with these properties,  set
$R^{-1}(v)=\{1,i_2,\ldots\}$ with $i_j<i_{j+1}$.  Consider the finite 
walk $R(2), R(3),
\ldots, R(i_2-1)$ on $G$.
By Remark \ref{avoiding}, this walk has an allowed walk on $G'$, which
by choice of $S$ occurs as a subwalk of $S$.  While $T_R$ sits on
$R(1)=v$, advance $T_S$ to the
beginning of this allowed sequence.  Then advance $T_R$ and $T_S$ until
$T_R$ can return to $v$, and repeat the process.
\end{pf*}
                                                         
The set of evasive walks arising from the construction of Theorem
\ref{K40} has measure 0, since each such walk misses some vertex of $G$.
We give a more elaborate construction to improve this situation, but 
first we introduce some more notation.  

\begin{dfn}
Denote by $Y_{abc}$ the graph consisting of a triangle with paths of
lengths $a$, $b$, $c$ attached to the vertices.  See Figure \ref{fig.yabc}.
\end{dfn}

\begin{figure}
\begin{center}
\psfrag{a}{\scalebox{2}{$a$}}
\psfrag{b}{\scalebox{2}{$b$}}
\psfrag{c}{\scalebox{2}{$c$}}
\scalebox{.5}{\includegraphics{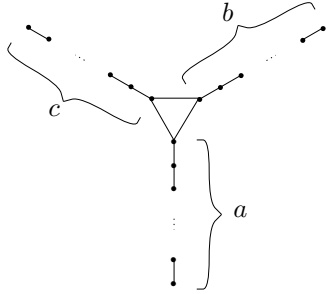}}
\caption{The graph $Y_{abc}$.}
\label{fig.yabc}
\end{center}
\end{figure}

For $G'$ a connected subgraph of
a graph $G$, let $$\langle G' \rangle _l$$ denote
any finite walk $S$ on $G'$ with the property that for every finite
walk $R$ on $G$ of length $\leq l$, $S$ contains a finite subwalk allowed for
$R$.  Such walks exist by Remark \ref{avoiding}, provided that $G'$
is not a path.  For vertices $v_1,\ldots,v_n$ we write $$\langle 
v_1\cdots v_n \rangle_l$$ to denote $\langle G' \rangle_l$ where
$G'$ is the subgraph of $G$ induced by the $v_i$.

For example,
on $K_4$ (with vertices $A$, $B$, $C$, $D$), the
symbol $\langle ABC \rangle _3$ could be used to denote the walk $ABCACBA$,
as this contains subwalks allowed for each of the $36$ possible walks of
length $3$ on $K_4$.

\begin{thm}\label{K4infty}
Let $G$ be a (finite, simple, connected) graph which is not
a tree, a cycle, or one of the graphs $Y_{abc}$.  Then there 
exist uncountably many evasive walks on $G$ which visit each
vertex of $G$ infinitely many times.
\end{thm}

For the proof we again use a characterization.

\begin{rem}\label{characterization.2}
Let $G$ be a (finite, simple, connected) graph.  Then $G$ is not a tree, a
cycle, or one of the graphs $Y_{abc}$ if and only if there exists a vertex
$v$ of $G$ such that $v$ is connected by at least two edges to some 
non-path component $G'$ of $G-v$.
\end{rem}

\begin{pf*}{Proof of Theorem \ref{K4infty}}
Choose $v$ and $G'$ according to Remark \ref{characterization.2}, and 
let $w\in G'$ be a neighbor of $v$.  Set 
$$\tilde S=w\langle G'
\rangle_{l_{11}}wvw\langle G' \rangle_{l_{21}}\langle G'
\rangle_{l_{22}}wvw \cdots,$$
with the integers $l_{ij}$ to be determined later.  (In order to ensure
that $S$ is a walk, we extend the ends
of each $\langle G' \rangle_{l_{ij}}$ if necessary so that they
start and end at a neighbor of $w$ in $G'$.)  If $G-v$ is connected,
let $S=\tilde S$; otherwise let $S$ be obtained from $\tilde S$ by 
replacing each occurrence of $v$ in $\tilde S$ with a finite walk
on $G$ which starts and ends at $v$ and which visits each vertex
of $G-G'$.
We show that for suitable choices of the $l_{ij}$, 
$S$ is an evasive walk.

Again let $n$ be the number of vertices of $G$; then with probability
$1/n$ a random walk $R$ starts at $v$ and visits it
infinitely often.  For such an $R$, let $R^{-1}(v)=\{i_1,i_2,\ldots\}$
where $i_j<i_{j+1}$ (thus $i_1=1$), and let $d_j=i_{j+1}-i_j-1$.
Likewise, let $S^{-1}(v)=\{n_1, n_2, \ldots \}$
denote the successive visits of $S$ to $v$.

Let $p_l= Pr(d_j \leq l)$ and $q=Pr\{R(i_j +1)\not= w \mbox{ and }
R(i_{j+1}-1) \not= w \}$; these are independent of $j$.

To begin, $T_R$ is on $R(1)=v$ and $T_S$ is on $S(1)=w$.  We say that the
tokens are {\em reset} initially, and again whenever $T_R$ is on $v$ and
$T_S$ is on $S(i+1)=w$, where $S(i)=v$.  We compute a lower bound on the
probability of being able to reset the tokens.

Case 1:  If $R(2)\not=w$ and $R(i_2-1)\not=w$, then with $T_R$ held on
$v$ we can advance $T_S$ through $\langle G' \rangle_{l_{11}}$ to
$S(n_1-1)=w$, move $T_R$ to $R(2)\not=w$, $T_S$ to
$S(n_1)=v$, $T_R$ to $R(i_2-1)\not=w$, $T_S$ to $S(n_1+1)=w$,
and $T_R$ to $R(i_2)=v$, resetting the tokens.  Note that this case 
includes the possibility that $R(2),\ldots,R(i_2-1)$ lie in a 
component of $G-v$ other than $G'$.

Case 2:  If $R(2)=w$ or $R(i_2-1)=w$ but $d_1\leq l_{11}$, $R(i_2+1)\not=w$,
and $R(i_3-1)\not=w$ then with $T_R$ on $R(1)=v$ we move $T_S$ to the beginning
of an allowed walk in $\langle G' \rangle_{l_{11}}$ for $R(2)R(3)\cdots
R(i_2-1)$.  We then advance $T_R$ to $R(i_2)=v$ and $T_S$ through
the remainder of $\langle G' \rangle_{l_{11}}$ to $S(n_1-1)=w$.  Then as
before, move $T_R$ to $R(i_2+1)\not=w$, $T_S$ to $S(n_1)=v$, $T_R$
to $R(i_3-1)\not=w$,  $T_S$ to $S(n_1+1)=w$, and $T_R$
to $R(i_3)=v$, resetting the tokens.

The probability of Case 1 or 2 occurring is
at least $q + (1-q)qp_{l_{11}}$.

We now try to reset the tokens again.  Note that
$T_R$ is now on $R(i_2)$ or $R(i_3)$; as it makes no difference for the
argument, we suppose it is on $R(i_2)$.
Again, the tokens can be reset if $R(i_2+1)\not=w$ and $R(i_3-1)\not=w$
(as in Case 1); or if $R(i_2+1)=w$ or $R(i_3-1)=w$, but $d_2 \leq l_{21}$,
$R(i_3+1)\not=w$, and $R(i_4-1) \not=w$ (as in Case 2, using an allowed
walk from $\langle G' \rangle_{l_{21}}$).

The new case is that $R(i_2+1)=w$ or $R(i_3-1)=w$, and $R(i_3+1)=w$ or
$R(i_4-1)=w$, but $d_2 \leq l_{21}$, $d_3 \leq l_{22}$, $R(i_4+1)\not=w$,
and $R(i_5-1)\not=w$.  Here we proceed similarly, but we must use
allowed walks from both $\langle G' \rangle_{l_{21}}$ and $\langle
G' \rangle_{l_{22}}$.

These three cases occur with probability at least
$q+(1-q)qp_{l_{21}}+(1-q)^2qp_{l_{21}}p_{l_{22}}$.

Continuing, we see that we can reset the tokens infinitely many times
with probability
$$p\geq\big(q+(1-q)qp_{l_{11}}\big)\big(
q+(1-q)qp_{l_{21}}+(1-q)^2qp_{l_{21}}p_{l_{22}}\big)\cdots,$$
thereby making $(R,S)$ a successful pair if this product converges.
The reader can verify that if each $p_{l_{ij}}$ were equal to 1, then
the product would
converge (recall that $q>0$ is constant).  Thus, since $p_l\to 1$ as
$l\to\infty$, we can choose the $l_{ij}$ so that $S$ is evasive, and
the conclusion follows.
\end{pf*}

The set of evasive walks arising from the construction of Theorem
\ref{K4infty} still has measure 0, since each such walk visits the
vertex $v$ with upper density zero.  (Recall that the {\em upper 
density} of a set of integers $X$ is 
$\limsup_{n\to\infty}{|X\cap\{1,\ldots,n\}|\over n}$.)  This is 
nevertheless the best we can do on $K_4$.  However, on
$K_n$ for $n\geq 5$ we have the following improvement.  

\begin{thm}\label{K5}  For $n\geq 5$, there exist (uncountably many)
evasive walks $S$ on $K_n$ for which $S^{-1}(v)$ has upper density
at least $1/n$ for each vertex $v$ of $G$.
\end{thm}

\begin{proof}
Denote the vertices of $K_n$ by $v_1,\ldots,v_n$.  Let $$S=
\langle v_1v_2v_3 \rangle_{l_1}
\langle v_2v_3v_4 \rangle_{l_2}
\cdots
\langle v_kv_{k+1}v_{k+2} \rangle_{l_k}
\cdots, $$
with the $l_i$ yet to be determined.  We will show that the $l_i$ can be 
chosen so that $S$ is evasive.

Choose $R$ at random, subject to $R(1)=v_{n-1}$.  Let
$i_1=1$ and $i_j=\min\{i>i_{j-1}: R(i)=v_{j-2}\}$ for $j=2,3,\ldots$, 
taking the indices mod $n$; thus $\{R(i_j)\}$ is the earliest 
subsequence of $R$ which looks like $v_{n-1}v_nv_1v_2v_3v_4\ldots\ $.
Let $d_j=i_{j+1}-i_j-1$ be the length of the subwalk between $R(i_j)$ 
and $R(i_{j+1})$.

We claim that if $d_j \leq l_j$ for all $j$, then $(R,S)$ is a 
successful pair.  To see this, assume $d_1 \leq l_1$.
Then $\langle v_1 v_2 v_3 \rangle_{l_1}$ contains an allowed subwalk for 
the subwalk $R(2), R(3), \cdots , R(i_2 -1)$.  Holding $T_R$ at 
$R(1)=v_{n-1}$, we advance $T_S$ 
to the start of this allowed walk.  We next advance $T_R$ and $T_S$ 
until $T_R$ arrives at $R(i_2)=v_n$, and finally, keeping $T_R$ fixed, 
we advance 
$T_S$ to the end of $\langle v_1, v_2, v_3 \rangle_{l_1}$.  

This process can be repeated indefinitely provided that each $d_i\leq l_i$;
this establishes the claim.  

Now, for fixed $j$, we have
$Pr(d_j\leq l_j)= 1-(\frac{n-2}{n-1})^{l_j}$.  
Thus $$Pr((R,S) \mbox{ is a successful pair})\geq \prod_j
\left(1-\left(\frac{n-2}{n-1}\right)^{l_j}\right),$$
a product which converges whenever $\sum_j (\frac{n-2}{n-1})^{l_j}$
does.  Thus $S$ is evasive for any choice of $\{l_j\}$ which makes the
above sum converge.

To verify the claim about upper densities, we choose 
$l_j=\lceil j/n \rceil$.  The infinite sum clearly
converges, and by symmetry we may choose $S$ so that each block
$$ \langle v_1 v_2 v_3 \rangle _k \langle v_2 v_3 v_4 \rangle _k \cdots 
\langle v_n v_1 v_2 \rangle _k $$ 
visits each vertex the same number of times.
\end{proof}

An understanding of how the length of $\langle v_1v_2v_3 \rangle_k$ grows
as a function of $k$ might provide information about the lower densities
of the sets $S^{-1}(v)$.

\section{Collisions on Cycles}\label{sec.cycle}

In this section we show that on a cycle, a collision will occur quickly
with high probability.  This implies that there are no evasive walks on 
a cycle.  The essential ingredient will be a winding number.

\begin{dfn}
Let $G$ be a cycle on $n$ vertices.  With $(R,S)$ a pair
of infinite walks on $G$, define the {\em fractional winding sequence} of $R$ 
by $f_R(1)=0$ and $$f_R(i+1)=\cases{ f_R(i)+1/n \mbox{ if $R(i+1)$ is a
clockwise step from $R(i)$, or }\cr
f_R(i)-1/n \mbox{ if $R(i+1)$ is a counter-clockwise step from $R(i)$.}}$$
Define $f_S(i)$ similarly, except that $f_S(1)=k/n$ where $0 \leq k < n$ 
is the number of clockwise steps from the start of $R$ to the start of $S$.  
The {\em winding number of $R$ at $i$} is $w_R(i)=\lfloor f_R(i) \rfloor$, 
and likewise for $S$. 
Define the {\em winding sequence} of $R$, $W_R$, to be the subsequence of 
$w_R(1),w_R(2), \ldots$ that ignores repetition; i.e. a maximal 
subsequence of $w_R(1),w_R(2), \dots$ such that consecutive elements 
are distinct.  Define $W_S$ likewise.
\end{dfn}

The winding numbers for a successful pair are related:

\begin{rem}\label{winding}
For $(R,S)$ a pair of walks on a cycle, if the clairvoyant demon can 
advance $T_R$ to $R(i)$ and $T_S$ to $S(j)$ then 

$$w_R(i) \leq w_S(j) \leq w_R(i) + 1.$$
\end{rem}

The following lemma is used in the proof of Theorem \ref{cyclebound}.
We postpone its proof until the end of this section.

\begin{lem}\label{roam}
A random walk of length $N$ on the integers stays within
$N^{1/3}$ of its starting point with
probability less than $e^{-cN^{1/3}},$ for some constant $c>0$.
\end{lem}

\begin{thm}\label{cyclebound}
Let $G$ be a cycle with $n$ vertices. Then there exists a
constant $c>0$ such that the probability that a random pair $(R,S)$
can be advanced $N$ steps without collision is less than
$e^{\textstyle -cN^{\frac{1}{3}}}$.
\end{thm}

\begin{proof}
Let $R$ and $S$ be walks on $G$.
Since $f_R(i)$ is a random walk on the integer multiples of $1/n$, Lemma
\ref{roam} provides a constant $c_0>0$ such that $$
Pr\left(|w_R(i)|\leq \frac{N^{1/3}}{n} \quad\forall i\leq N\right) 
\leq e^{-c_0N^{1/3}}.$$
Ignoring these cases, we find that the winding number exceeds 
$\frac{1}{n}N^{1/3}$ 
(for some $i\leq N$) with probability at least $1-e^{-c_0N^{1/3}}$.  
(If $R$ achieves $-\frac{1}{n}N^{1/3}$ as a winding number instead,
we replace all winding numbers by their negatives in the ensuing argument.)

We will say that the winding sequence of a walk {\em decreases at $k$} 
if subsequent to the first occurrence of $k+1$, the next occurrence of 
$k-1$ is before the first occurence of $k+2$, and {\em increases strongly 
at $k$} if from the first occurrence of $k$ it proceeds $k, k+1, k+2$.

By Remark \ref{winding}, any scheduling of the tokens has the property 
that $W_S\leq k$ until the first time $W_R$ reaches $k$ (at which time 
$W_S=k$).
The key observation is that, again by Remark \ref{winding}, the
pair $(R,S)$ is not successful if there is a $k$ at which $W_R$ increases 
strongly while $W_S$ decreases. 
In such a case, collision of the tokens is ensured by the time $W_R$ 
reaches $k+2$.  Let $E_k$ be the event that $W_R$ 
increases strongly at $k$ and $W_S$ decreases at $k$.  
Now, there is some $p>0$ (depending on $n$)
such that $Pr(E_k)>p$ for all $k$.  
For the walk to be free of collision for more than $N$ steps, none of the 
events $E_k$ may occur for $0 \leq k \leq \frac{1}{n}N^{1/3}$.  If the events 
$E_k$ were
independent, then the probability of this would be less 
than $(1-p)^{\frac{1}{n}N^{1/3}}=e^{-cN^{1/3}}$.  
However, while the decrease of $W_S$ at $k$ is independent of $k$, 
the strong increase of 
$W_R$ is not, since its occurrence at $k$ increases the chance of its 
occurrence at $k+1$.  Notice, however, that this dependence only helps in the 
inequality.  (Alternatively, one can consider the events $E_k$ for even $k$, 
which are independent.  This produces a change in the constant $c$.)
We therefore have the desired result.
\end{proof}

\begin{thm}\label{noevasive}
There are no evasive walks on a cycle.
\end{thm}

\begin{proof}  Let $S$ be a walk on a cycle.
If $W_S$ decreases at only a finite number of integers then $W_S$ is 
bounded below.  The pair $(R,S)$ can then be successful only if $W_R$ is 
bounded below as well, an event of probability zero. 

In the remaining case $W_S$ decreases at an infinite set of integers $T$.  
By the argument of Theorem \ref{cyclebound}, for $(R,S)$ to be a successful 
pair, $W_R$ cannot increase strongly at any of the integers in $T$.  
As $T$ is infinite, the probability of this event is zero.  Consequently 
$S$ is not an evasive walk. 
\end{proof}

\begin{proof}[Proof of Lemma \ref{roam}] Let $G$ be a single path with $l$ 
nodes.  It is known that the adjacency matrix $M$ for $G$
has norm $2 \cos ( \pi /{(l+1)})$ (\cite{book}, p.21).
Set $l = 2r + 1$ and let $e_{r+1}$ be the vector in $\R^l$ with $1$ in
the middle
position and zeros elsewhere.  Let $v = M^N e_{r+1} = (v_1, v_2, \ldots,
v_l).$
Then $\sum_{1\leq i \leq l }v_i$ is the total number of walks on the line
of length $N$ which stay within $r$ of the origin.  Since
$$\sum_{1\leq i \leq l }v_i^2 \leq (2 \cos ( \pi /{(l+1)}))^{2N}, $$
we have
$$ \sum_{1\leq i \leq l }v_i \leq \sqrt{l} (2 \cos ( \pi
/{(l+1)}))^{N}.$$
Since there are $2^N$ total paths of length $N$, the probability that a
random walk of length $N$ stays within $r$ of the origin is less than
$\sqrt{l}(\cos (\pi / (l+1)))^N$.  Setting $r= N^{1/3}$ and using the
first three terms of the Taylor expansion for cosine yields
$$  \sqrt{l}(\cos (\pi / (l+1)))^N  \leq e^{-cN^{\frac{1}{3}}}. $$
\end{proof}

\section{Variants}\label{sec.variants}

The general task of scheduling tokens to avoid collision has natural
interpretations in areas such as traffic control and robotics
(see for example \cite{robotics}).
In these contexts one might consider further constraints on the motions
of the tokens---for instance by making certain edges ``one-way" or
by restricting a token to motion on a subgraph.  
In this section we address two variants of the clairvoyant demon
problem which incorporate these constraints.  These variants were
suggested to us by Peter Winkler.  We note that the proofs
of Propositions \ref{prop.oneway} and \ref{prop.purple} provide the
demon with algorithms for solving these scheduling problems.

\subsection{Directed graphs}\label{sec.directed} \label{sec.oneway}

Let $G$ be a finite, connected, simple graph whose edges are partitioned
into two sets, $\mathcal O$ and $\mathcal T$.  For each $e\in \mathcal O$
choose an orientation of $e$; this $e$ is a {\em one-way} edge.  Edges of
$\mathcal T$ are unoriented; they are called {\em two-way} edges.
We call this structure a {\em directed graph}.  Tokens on a directed
graph are allowed to travel along two-way edges as before, but only
in the direction of the orientation on one-way edges.

Define $\Gamma _G$ to be the subgraph of $G$ consisting of all the 
vertices of $G$ and all the two-way edges of $G$.  (Note that $\Gamma_G$ 
might be disconnected.)

\begin{prop} \label{prop.oneway} If $G$ is a directed graph such that every 
connected component of $\Gamma_G$ has a directed edge in $G$ leading to a 
different connected component of $\Gamma_G$, then the clairvoyant demon 
succeeds on $G$.
\end{prop}

\begin{proof}  We show that the clairvoyant demon succeeds in the case 
where (a) the two tokens $S$ and $T$ start in different connected components 
of $\Gamma_G$ and (b) each token's path does not remain in a single 
connected component of $\Gamma_G$ for an infinite number of consecutive 
steps.  This happens with positive probability since the probability of 
(a) is positive and the probability of (b) is 1.

In this case, the demon has a successful strategy as follows.  
Advance $S$ $100$ steps or until it is
about to move into the connected component of $\Gamma_G$ where $T$
currently is. Then advance $T$ $100$ steps or until it is about to move into the
connected component of $\Gamma_G$ where $S$ is now located.  Repeat these 
steps until each token is about to move into the other token's connected 
component. 
At this point, the tokens are not about to switch vertices,
since the vertices on which they sit are not connected by a two-way
edge. Hence, it will be possible to advance $S$ and $T$ one step each
in one order or the other.  After doing this, $S$ and $T$ are once again
in different connected components of $\Gamma_G$.  Continuing in
this manner, we may advance both tokens forever.
\end{proof}

\subsection{Colored Graphs}\label{sec.purple}

Let $G$ be a (finite, simple, connected) graph with edge set $E=E_R\cup E_B$.
Edges in $E_R$ are \emph{red}, edges in $E_B$ are \emph{blue}, and edges
in $E_P=E_R \cap E_B$ are \emph{purple}.  If both subgraphs $G_R$ (formed by 
the edges $E_R$) and $G_B$ (formed by the edges of $E_B$)
are connected, then we call this structure a
\emph{colored graph}.  We allow token $S$ to move only on red edges and
token $T$ to move only on blue edges.  (This is most interesting when $E_P$
is large, e.g. $E_P=E_B$.)  We define the purple subgraph $G_P$ to consist
of all vertices of $G$ together with the edge set $E_P$.  (Thus
$G_P$ may be disconnected or even have isolated vertices.)

\begin{prop}\label{prop.purple}  Suppose $G$ is a colored graph such that 
neither $G_R$ nor $G_B$ is a path.  If $G_P$ is 
disconnected then the clairvoyant demon succeeds on the colored graph $G$.
\end{prop}

\begin{proof}
We treat two cases.

Case 1:  Each token can reach more than one connected component of $G_P$.
In this case we do not need the hypothesis about $E_R$ and $E_B$.
As in Proposition 
\ref{prop.oneway} we may assume that the two tokens $S$ and $T$ start in 
different connected components of $G_P$ and that each token's path does not 
remain in a single connected component of $G_P$ for an infinite number of 
consecutive steps.  The identical strategy to that of Proposition 
\ref{prop.oneway} ensures the demon's success.

Case 2:  At least one token, say $S$, has access to only one connected 
component $C$ of $G_P$.  Then $G_P$ must consist of $C$ together with the 
remaining 
vertices of $G$.  With positive probability, we can assume that $T$'s walk 
does not start in $C$ and does not remain in $C$ for an infinite 
number of consecutive steps.  Independently, with probability 1, every finite 
walk on $G_R$ occurs as a subwalk of $S$ 
infinitely many times.   In this case we proceed as in Theorem \ref{K40}:  
$T$ consists of a series (possibly infinite) of finite walks in $C$ separated 
by walks in the complement of $C$.  The clairvoyant demon succeeds by advancing 
token $S$ to the beginning of an allowed walk for the 
next visit of $T$ to $C$.  Then the demon can advance $T$ and $S$ together 
until $T$ leaves $C$ again.  Continuing in this way, both tokens 
can be advanced forever.\end{proof}

\section{Acknowledgements}
The authors would like to thank Peter Winkler, who introduced us to
the clairvoyant demon problem.  We also thank Julie Landau, without whose 
support and cross-court forehand this paper would not have been possible.


\bibliographystyle{amsplain}

\end{document}